\documentclass[11pt]{amsart}                          
\usepackage{epsfig}       
\usepackage{subfig}     
\usepackage{amssymb,amscd,bbm}     
\usepackage{amsthm,amsgen,amsmath,epic,psfrag}
\usepackage{rotating}
\usepackage{graphicx} 
\usepackage[all]{xypic}       
\usepackage{appendix}

\xyoption{curve}                
\xyoption{rotate}               

\newcommand{\R}{\mathbb R}
\newcommand{\I}{\mathbb I}
\newcommand{\ep}{\varepsilon}

\newcommand{\Emb}{\mathrm{Emb}}
\newcommand{\Imm}{\mathrm{Imm}}
\newcommand{\D}{\mathcal{D}}

\newcommand{\aaa}{\alpha\otimes\alpha\otimes\alpha}
\newcommand{\aaaa}{\alpha\otimes\alpha\otimes\alpha\otimes\alpha}
\newcommand{\ssss}{S^{d-1}\times S^{d-1}\times S^{d-1}\times S^{d-1}}
\newcommand{\M}{\mathcal{M}}
\newcommand{\Conf}{ev^*C}

\theoremstyle{plain}
\newtheorem{theorem}{Theorem}[section]
\newtheorem{proposition}[theorem]{Proposition}
\newtheorem{lemma}[theorem]{Lemma}

\theoremstyle{definition}
\newtheorem{definition}[theorem]{Definition}

\newtheorem{example}[theorem]{Example}

\theoremstyle{remark}

\begin{document}
\title{A geometric homology representative in the space of long knots}
\author{Kristine E. Pelatt}

\maketitle

\begin{abstract}
We produce explicit geometric representatives of nontrivial homology classes in $\Emb(\hat{S}^1,\R^d)$, the space of long knots, when $d$ is even.  We generalize results of Cattaneo, Cotta-Ramusino and Longoni to define cycles which live off of the vanishing line of a homology spectral sequence due to Sinha.  We use configuration space integrals to show our classes pair nontrivially with cohomology classes due to Longoni.  
\end{abstract}

\section{Introduction}

Knot spaces have recently been the subject of much interest.  Let $\Emb(\hat{S}^1,\R^d)$ be the space of embeddings from $S^1$ to $\R^d$ with fixed initial point and initial tangent vector, which is homotopy equivalent to the space of long knots.  Using Goodwillie-Weiss embedding calculus, Sinha \cite{Sinh02} defines spectral sequences converging to the homology and cohomology of $\Emb(\hat{S}^1,\R^d)$ for $d>3$.  Lambrechts, Turchin and Volic \cite{LTV06} have shown that the rational cohomology spectral sequence collapses at the $E_2$ page.  There is another spectral sequence, due to Vassiliev  \cite{Vass92}, which converges to the homology of $\Emb(S^1,\R^d)$.  The $E_1$ term of Vassiliev's spectral sequence agrees with the $E_2$ term of the embedding calculus spectral sequence by work of Turchin \cite{Tour07}.  These approaches allow one to combinatorially understand the ranks of the homology groups of $\Emb(S^1,\R^d)$, but do not immediately give geometric understanding or representing cycles and cocycles in knot spaces. We present representing cycles and cocycles defined through techniques which apply to all classes in the spectral sequence.

In \cite{CCL02} Cattaneo, Cotta-Ramusino and Longoni produce explicit, nontrivial, $k(d-3)-$dimensional cycles and cocycles.  We give a brief summary of these results in Section 3.  They define a chain map from a graph complex to the de Rham complex of $\Emb(S^1,\R^d)$, and produce cocycles as images of graph cocycles consisting of trivalent graphs.  To produce cycles, they use families of resolutions of singular knots with $k$ transverse double points.  These cycles all live along the $(-2q,q(d-1))-$diagonal in the first page of the homology spectral sequence, which also serves as a vanishing line. To establish nontriviality, they show the pairing between certain cycles and cocycles is nonzero.  For $d$ odd, Sakai produces a $(3d-8)-$dimensional cocycle in the space of long knots coming from a non-trivalent graph cocycle.  To establish the nontriviality of this cocycle, he evaluates it on a cycle produced using the Browder bracket coming from the action of the little two-cubes operad on the space of framed knots.  
 
The main result of this paper is the explicit production of a nontrivial cycle which lives off of the vanishing line of the homology spectral sequence for $d$ even, using techniques which should generalize.  We define this cycle by generalizing the methods of Cattaneo, Cotta-Ramusino and Longoni to families of resolutions of singular knots with triple points.  In particular, we first define a topological manifold $M_\beta$ and an embedding of $M_\beta$ into $\Emb(\hat{S}^1,\R^d)$, extending and correcting the results in a preprint of Longoni \cite{Long04}.  Longoni also defines a cocycle which is the image of a non-trivalent graph when $d$ is even.   We show that the pairing between Longoni's cocycle and our cycle is nonzero and thus both are nontrivial.   

Our cycle generalizes, and our techniques are closely related to the spectral sequence combinatorics, giving possible recipes for representatives of all cycles in the embedding calculus spectral sequence.   This is in contrast to Sakai's approach, which would require new input for any Browder-primitive classes off of the $(-2q,q(d-1))-$diagonal.  These results will appear in future work, but we discuss them briefly at the end of this paper. 

\section{Definition of the cycle}\label{sec:cycledef}

The idea at the heart of our method to produce homology classes in knot spaces goes back to Vassiliev's seminal work \cite{Vass92}.  In finite type knot theory, one defines the derivative of a knot invariant by taking an immersion with transverse double-points and evaluating the knot invariant on the resolutions of that immersion.  We require a generalization of such immersions.  

\begin{definition} 
An immersion $\gamma:S^{1} \hookrightarrow \R^{d}$ has a transverse intersection $r$-singularity at $\bar{t} = (t_1,t_2,\ldots,t_r)\in \I^{\times r}$ with $0< t_1<t_2<\cdots <t_r<1$, if all of the $\gamma(t_{i})$ coincide and the derivatives $\gamma'(t_{i})$ are generic in the sense that any $d$ or fewer of them are linearly independent. 
\end{definition}  

To connect with the language naturally produced by the embedding calculus spectral sequence, we use bracket expressions to encode singularity data.  Sinha calculates in \cite{Sinh02} that the 
subgroup of $\mathcal{P}ois^d(p)$, the $p-$th entry of the Poisson operad (see \cite{Sinh06.3}), generated by expressions with $q$ brackets such that each $x_i$ appears inside a bracket pair and the multiplication $``\cdot"$ does not appear inside a bracket pair, is also a subgroup of $E^1_{-p,q(d-1)}$ in the reduced homology spectral sequence.  This is the full $E^1_{-p,q(d-1)}$ in the spectral sequence converging to the homology of the space of embeddings modulo immersions.  On this subgroup, the differential $d_1: E^1_{-p,q(d-1)}\rightarrow E^1_{-p-1,q(d-1)}$ is $d^1=\sum_{i=0}^p (-1)^i(\delta^i)_*$, where $(\delta^0)_*$ is defined by adding $x_1$ in front of the expression and replacing each $x_j$ by $x_{j+1}$, $(\delta^{p+1})_*$ is defined by adding $x_{p+1}$ to the end, and for $1\leq i\leq p$, the map $(\delta^i)_*$  is defined by replacing $x_i$ by $x_i\cdot x_{i+1}$ and $x_j$ by $x_{j+1}$ for $j>i$.  In \cite{Tour07}, Tourtchine does further calculations in this spectral sequence. 

\begin{example} The bracket expression $\beta= \beta_1 +\beta_2$ where $\beta_1= \left[ [x_1,x_4],x_3\right] \cdot [x_2,x_5 ]$ and $ \beta_2= [x_1,x_4]\cdot \left[ [x_2,x_5],x_3\right]$ is a cycle in $E^1_{-5,3(d-1)}$.
\end{example}

\begin{definition} A pair $(\gamma, \bar{t})$ of an immersion and a sequence $\bar{t} = 0<t_1< t_2< \cdots < t_p<1$ respects a bracket expression $\beta \in \mathcal{P}ois^d(p)$ if $\gamma$ has a transverse $r$-singularity at the sequence $0<t_{i_1}< \ldots< t_{i_r}<1$ whenever $x_{i_1},\ldots,x_{i_r}$ appear inside of a bracket in $\beta$.  
\end{definition}  

For example, the knots $K_1$ and $K_2$ in Figure ~\ref{fig:K1K2} respect $\beta_1$ and $\beta_2$, respectively.  A knot can respect a bracket expression but have higher singularities; for example $K_1$ also respects $[x_1,x_3]\cdot [x_2,x_4]$.  
\begin{definition}We will denote the subspace of all pairs $(\gamma, \bar{t})\in \Imm(\hat{S}^1,\R^d)\times \I^{\times r}$ respecting a bracket expression by $\Imm_{\geq \beta}(\hat{S}^1,\R^d)$, with the convention $\Imm_\phi(\hat{S^1},\R^d) = \Imm(\hat{S^1},\R^d)$.  The subspace of $\Imm_{\geq \beta}(\hat{S}^1,\R^d)$ consisting of immersions which do not have higher singularities will be denoted by $\Imm_{=\beta} (\hat{S}^1,\R^d)$.  \end{definition}

\begin{figure}[h] 
   \centering
  $$ \includegraphics[width=3in]{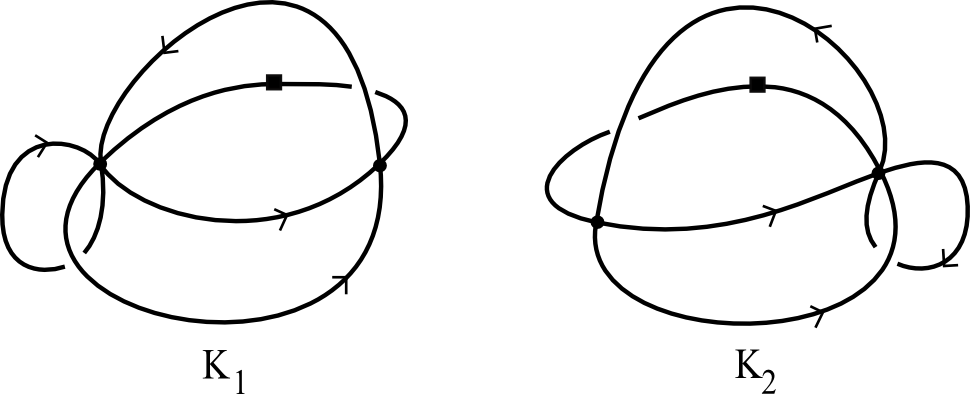} $$
   \caption{The singular knots $K_1$ and $K_2$.}
   \label{fig:K1K2}
\end{figure}

In the spectral sequence, bracket expressions of the form $\prod_{m=1}^k [x_{i_m},x_{j_m}] $ are $E^1$-cycles.  Submanifolds representing these cycles are well known and described in Section 2 of \cite{CCL02}.  Briefly, we start with a singular knot $K\subset \R^d$ with $k$ double points which respects $\prod_{m=1}^k [x_{i_m},x_{j_m}]$, and resolve each double point by moving one strand passing through the double point off of the other.  For each vector in $S^{d-3}$ we have a possible direction in which to move the strand, and therefore a possible way to resolve the double point.  The subset of $\Emb(\hat{S}^1,\R^d)$ consisting of all such resolutions of $K$ is a submanifold parameterized by $\prod_m S^{d-3}$, and its fundamental class corresponds to the cycle $\prod_{m=1}^k [x_{i_m},x_{j_m}] $ of the spectral sequence.  

For higher singularities, we start with ideas of Longoni \cite{Long04} and produce resolutions of  transverse intersection singularities by moving one strand at a time off the intersection point.    Assume the rank of the singularity $r$ is less than $d$, so the (tangent vectors of the) strands in question span a proper subspace.  There are two cases - resolving a double point and resolving a higher singularity.  If $r\geq 3$, we are moving a strand off the intersection point.  The complementary subspace to the (tangent vector of the) strand has a unit sphere $S^{d-2}$ which parametrizes the directions to move one strand off the intersection point.  If $r = 2$, we consider a unit sphere $S^{d-3}$ in the complimentary subspace which parametrizes the directions to move one strand off another.  

Resolutions of triple point singularities (and higher singularities) can produce further singularities (see Figure \ref{fig:K3to6}).  By restricting away from neighborhoods of those ``additional singularity'' resolutions, we produce submanifolds with boundary which we show can be pieced together to build representatives of $E^1$-cycles in the spectral sequence.  We formalize as follows.

\begin{definition}
If $\beta$ is a bracket expression, let $\beta(\hat{i})$ denote the bracket expression obtained from $\beta$ by removing $x_i$ and the minimal set of other symbols as required to have a bracket expression, and replacing $x_k$ by $x_{k-1}$ for all $k>i$.  
\end{definition}

For example, with $\beta_1= \left[ [x_1,x_4],x_3\right] \cdot [x_2,x_5 ]$, we have $\beta_1(\hat{4}) = [x_1,x_3][x_2,x_4]$.  For each strand through a transverse intersection $r$-singularity, we can define a resolution map which moves that strand off of the singularity.  To accommodate the two cases, we let \[d(r) = \left\{ \begin{array}{cc} d-2 & \mathrm{if}\; r> 2\\ d-3 &\mathrm{if}\; r= 2\end{array}\right..\]  

By the rank of $x_i$ in a bracket expression $\beta$, we will mean the number of variables in $\beta$ (counting $x_i$) which appear inside of common brackets with $x_i$.  In $\beta_1$, $x_3$ has rank three and $x_5$ has rank two. 

\begin{definition}
If $\beta$ is a bracket expression in which $x_i$ has rank $r$ (with $r > 0$)
define the resolution map \[\rho_{i} : \Imm_{\geq \beta}(\hat{S}^1,\R^d) \times S^{d(r)} \times \I \times \I
\to \Imm_{\geq \beta(\hat{i})}(\hat{S}^1,\R^d)\] by \[\rho_i(\gamma,\bar{t},v,a,\varepsilon)(t) = \left\{ \begin{array}{cc} \gamma(t) + a\cdot v  \exp\left( \frac{1}{(t-t_i)^2-\varepsilon^2}\right) & \mathrm{if}\; t\in(t_i-\varepsilon,t_i+\varepsilon)\\ \gamma(t) & \mathrm{otherwise} \end{array} \right. .\] \end{definition}

We call the triple $(v, a, \varepsilon) \in S^{d(r)} \times \I \times \I$ the resolution data.  We often fix $a$ and $\varepsilon$ so that the resolutions do not have unexpected singularities and by abuse denote the restriction by $\rho_{i}$ as well.  The resolution map produces immersions in which the strand (between times $t_i - \ep$ and $t_i + \ep$) is moved in the direction of $v$, as shown in Figure ~\ref{fig:resolution}.

\begin{figure}[h] 
   \centering
  $$ \includegraphics[width=2.2in]{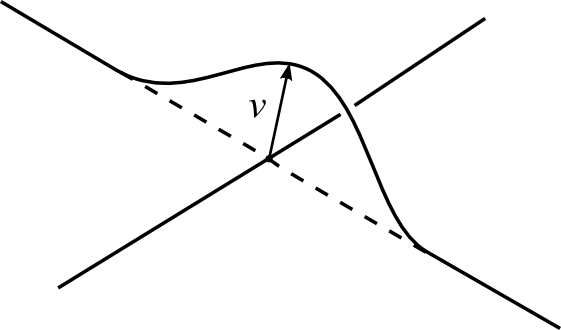} $$
   \caption{The resolution of a double point.}
   \label{fig:resolution}
\end{figure}

\begin{definition}\label{def:resmap}
Let $S = \{ x_{i_1},x_{i_2},\ldots,x_{i_k}\}$ be an ordered subset of the variables in $\beta$.  Define $\rho_{\beta, S}$ to be the composite 
$$\rho_{i_{k}} \circ (\rho_{i_{k-1}} \times id)  \circ \cdots \circ (\rho_{i_{1}} \times id)  : 
\Imm_{\geq \beta}(\hat{S}^1,\R^d) \times \prod_{m} \left(S^{d(r_{m}) }\times \I \times \I\right) \to \Imm_{\geq \varnothing}(\hat{S}^1,\R^d),$$
where $r_{m}$ is the rank of $x_{i_{m}}$ in $\beta(\hat{i_{1}}, \ldots, \hat{i}_{m-1})$.
\end{definition}

The set $S$ encodes which strands get moved in the resolution defined by $\rho_{\beta,S}$.  

We now specialize.  Let $\beta_{1} =  \left[ [x_1,x_4],x_3\right] \cdot [x_2,x_5 ]$, $ \beta_2= [x_1,x_4]\cdot \left[ [x_2,x_5],x_3\right] $ and choose the ordered subset of variables for each to be $S=\{x_3,x_4,x_5\}$.  We choose embeddings $K_1$ and $K_2$ of $S^1$ in $\R^3\hookrightarrow \R^d$ as shown in Figure ~\ref{fig:K1K2}, as well as a sequence $0<t_1<t_2<\cdots<t_5<1$ so that $(K_1,\bar{t})$ respects $\beta_1$ and $(K_2,\bar{t})$ respects $\beta_2$.

We restrict the directions in which the singularities are resolved to ensure we produce not just immersions but embeddings.  We assume that in the disk of radius $1/10$ centered at each singularity, both $K_1$ and $K_2$ consist of linear segments intersecting transversely, as shown in Figure ~\ref{fig:disks}.  Fix $\ep>0$ so that the intervals $[t_i-\ep,t_i+\ep]$, $i=1,2, \ldots,5$, are disjoint and $K_1([t_i-\ep,t_i+\ep])$ is contained in $B_{\frac{1}{10}}(K_1(s_i))$ for $i=1,2,\ldots,5$.   These intervals are the strands we will move to resolve the singularities.  

Let $w_1,\ldots, w_5$ be the unit tangent vectors to each line segment at the singular points of $K_1$.  Fix $\delta >0$ so that $\{v\in S^{d-2} :\;\parallel v-w_1\parallel<\delta\}$ and $\{v\in S^{d-2} :\;\parallel v-w_4\parallel<\delta\}$ are disjoint.  As mentioned above, we avoid moving the third strand off of the triple point in these directions to prevent the introduction of a double point.

\begin{figure}[h]
\begin{center}
$$\includegraphics[width=4in]{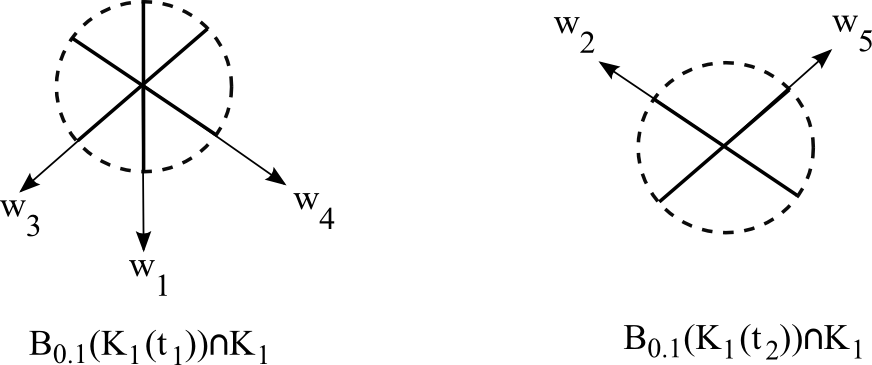}$$
\caption{$B_{\frac{1}{10}}(K_1(t_1))\cap K_1$ and $B_{\frac{1}{10}}(K_1(t_2))\cap K_1$.}
\label{fig:disks}
\end{center}
\end{figure}

We produce a manifold $\M_\beta$ as the image of a topological manifold $M_\beta$ embedded in $\Emb(\hat{S}^1,\R^d)$ by resolving singular knots with triple and double points.  The manifold $\M_\beta$ decomposes as the union $\bigcup_{i=1}^6 \M_i$, where each $\M_i$ is the image in $\Emb(\hat{S}^1,\R^d)$ of a resolution map defined below.  The domains of the resolution maps for the main pieces, $\M_1$ and $\M_2$, are denoted $M_1$ and $M_2$ and are homeomorphic to $\left( S^{d-2} \setminus \cup_4 B_\delta\right) \times S^{d-3}\times S^{d-3}$.   The domains of resolution maps defining the remaining four families are denoted $M_i\times \I$, where $M_i$ is homeomorphic to $S^{d-3}\times S^{d-3}\times S^{d-3}$ for $i=3,4,5,6$.  

\begin{definition}
For any triple $(\ep_3,\ep_4,\ep_5)$ with each $\ep_i\leq \ep$ for $\ep$ as above, define \[M_{1}(\ep_3,\ep_4,\ep_5) \subset  \Imm_{\geq \beta_{1}}(\hat{S}^1,\R^d) \times \prod_{k=3}^5 (S^{d(r_{k}) }\times \I \times \I)\] as the subspace of all
$K_{1} \times \prod (v_{i}, a_{i}, \varepsilon_i),$ where $a_3 = \frac{1}{10}$, $a_4=a_5 = \frac{\delta}{10}$,  and $v_3$ is such that the distances between $v_3$ and the vectors $\pm w_1$ and $\pm w_4$ are all greater than or equal to $\delta$.  There are no restrictions on $v_4,v_5\in S^{d-3}$.
\end{definition}

We will suppress the dependence of $M_1$ on the values of $\ep_3,\ep_4,\ep_5 \leq \ep$ as well as $\delta$ except when needed. 

\begin{lemma}
The restriction of $\rho_{\beta_1,S}$ to $M_{1}$ maps to $\Emb(\hat{S}^1,\R^d) \subset \Imm_{\geq \phi}(\hat{S}^1,\R^d)$.
\end{lemma}

Choose the immersion $K_2$ as shown in Figure \ref{fig:K1K2}, and assume that the constants $\delta >0$ and $\varepsilon>0$ chosen above satisfy similar conditions for $K_2$,  to define $M_2$ analogously. The restriction of $\rho_{\beta_2,S}$ maps $M_2$ to  $\Emb(\hat{S}^1,\R^d) \subset \Imm_{\geq \phi}(\hat{S}^1,\R^d)$.  We denote the families of embeddings $\rho_{\beta_1,S}(M_1)$ and $\rho_{\beta_2,S}(M_2)$ by $\M_1$ and $\M_2$ respectively, and connect the boundary components of $\M_1$ to those of $\M_2$ to build a family without boundary.  

Each boundary component can also be described as the family of knots obtained by resolving a singular knot with three double points.  In fact, resolving the triple point in $K_1$ by moving the strand $K_1\left([t_3-\varepsilon_3, t_3+\varepsilon_3]\right)$ in the direction of $\pm w_1$ or $\pm w_4$ yields an immersion with three double points.  The four boundary components of $\M_1$ are families of resolutions of these four knots.

\begin{definition}  Let $K_3, K_4, K_5$ and $K_6$ be the singular knots, each with three double points, defined below and shown in Figure \ref{fig:K3to6}.   
\begin{align*} K_3 & = \rho_{3}\left(K_1, w_4, \tfrac{1}{10}, \ep_3 \right) \\
K_4 & = \rho_{3}\left(K_1, -w_4, \tfrac{1}{10}, \ep_3\right)\\
K_5 & = \rho_{3}\left(K_1, w_1,\tfrac{1}{10}, \ep_3\right) \\
K_6 & = \rho_{3}\left(K_1,-w_1, \tfrac{1}{10}, \ep_3 \right)  \end{align*} \end{definition}

\begin{figure}[h]
\begin{center}
$$\includegraphics[width=3in]{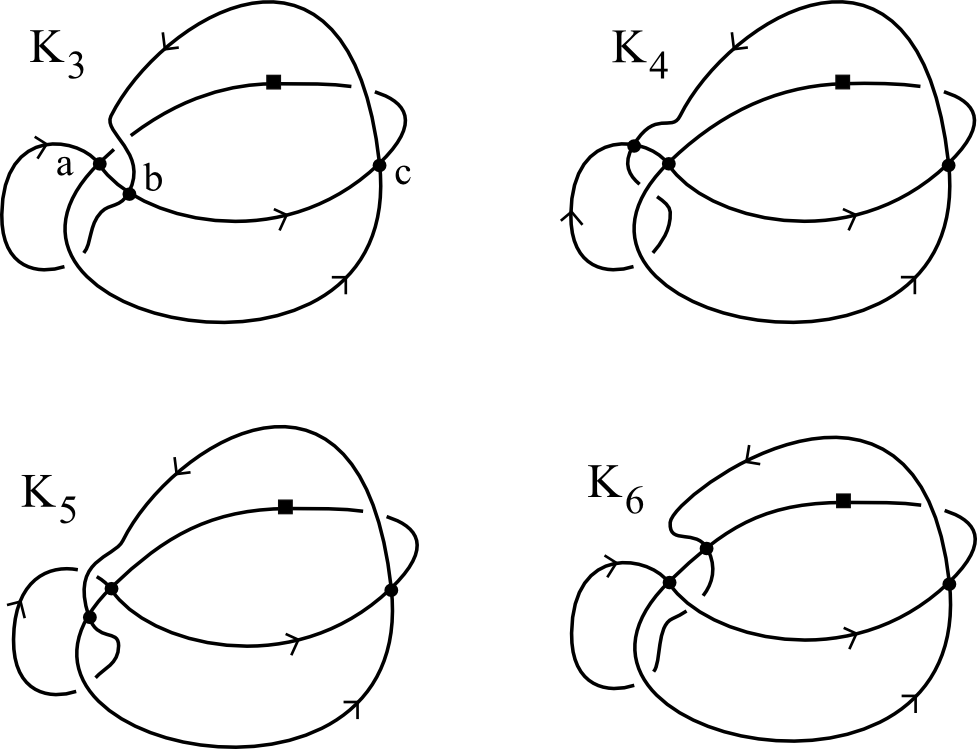}$$
\caption{Singular knots $K_3$, $K_4$, $K_5$, and $K_6$.}
\label{fig:K3to6}
\end{center}
\end{figure}

We resolve these knots, restricting the directions so the resulting embeddings are those in the boundary components of $\M_1$.  Initially, we focus on $K_3$.  The double points corresponding to $[x_1,x_4]$ and $[x_2,x_6]$, labeled $a$ and $c$, are resolved in the same way as the double points in $K_1$.  The double point corresponding to $[x_3,x_5]$, labeled $b$, is resolved using only vectors in the direction $v-w_4$ for some $v$ such that $\parallel v - w_4 \parallel = \delta$.  This guarantees that resolving this double point in $K_3$ yields the $\parallel v_3 - w_4\parallel = \delta$ boundary component of $\M_1$.  

\begin{definition}
Define $M_3(\ep_3,\ep_4,\ep_5) \subset \Imm_{\geq \beta_3}(\hat{S}^1,\R^d) \times \prod_{i=3,4,6}\left( S^{d-3}\times \I\times \I\right)$ where $\beta_3=[x_1,x_4]\cdot[x_2,x_6]\cdot [x_3,x_5]$ as the subset of all $K_3\times \prod_{i = 3,4,6}\left(u_i,\frac{\delta}{10},\varepsilon_i\right)$ where $u_4$ and $u_6$ are unrestricted and $u_3$ satisfies $\parallel w_4 +\delta u_3\parallel = 1$. 
\end{definition} 

\begin{proposition}\label{prop:resolboundary}
Let  $S_3 = \{x_3,x_4,x_6\}$. The restriction of $\rho_{\beta_3, S_3}$ maps $M_{3}$ to $\Emb(\hat{S}^1,\R^d) \subset \Imm_{\geq \phi}(\hat{S}^1,\R^d)$, and $\rho_{\beta_3,S_3}(M_3)$ is the $\parallel v_3 - w_4\parallel = \delta$ boundary component of $\M_1$.
\end{proposition}

\begin{proof}
The resolution $\rho_{\beta_3,S_3}(K_3) = \rho_{\beta_3}\left( \rho_{3}\left( K_1,w_4,\frac{1}{10},\ep_3\right) \right)$ using $u_3$ as in the definition of $M_3(\ep_3,\ep_4,\ep_5)$ is the same embedding as the resolution $\rho_{\beta_1} (K_1)$ using $v_3 = w_4+\delta u_3$, since \[\frac{1}{10} w_4\,\exp \left(\frac{1}{(t-t_3)^2+\ep_3^2}\right) + \frac{\delta}{10}u_3  \exp\left(\frac{1}{(t-t_3)^2+\ep_3^2}\right) = \frac{1}{10} v_3 \exp \left(\frac{1}{(t-t_3)^2+\ep_3^2}\right).\] \end{proof}

Similarly resolving the knots $K_4, K_5$, and $K_6$ yields the boundary components of $\M_1$ corresponding to $\parallel v_3 + w_4\parallel = \delta$, $\parallel v_3 - w_1\parallel = \delta$, and $\parallel v_3 + w_1\parallel = \delta$ respectively.  This process can also be applied to the boundary components of $\M_2$.  Let $K_7, K_8, K_9,$ and $K_{10}$ be the four singular knots obtained from $K_2$ by moving $K_2\left([t_3-\ep_3,t_3+\ep_3]\right)$ in the direction of the tangent vectors to the other two strands intersecting at the triple point, as shown in Figure \ref{fig:k7to10}.  As with $K_1$, resolving these singular knots gives the four boundary components of $\M_2$.

\begin{figure}[h]
\begin{center}
$$\includegraphics[width=3in]{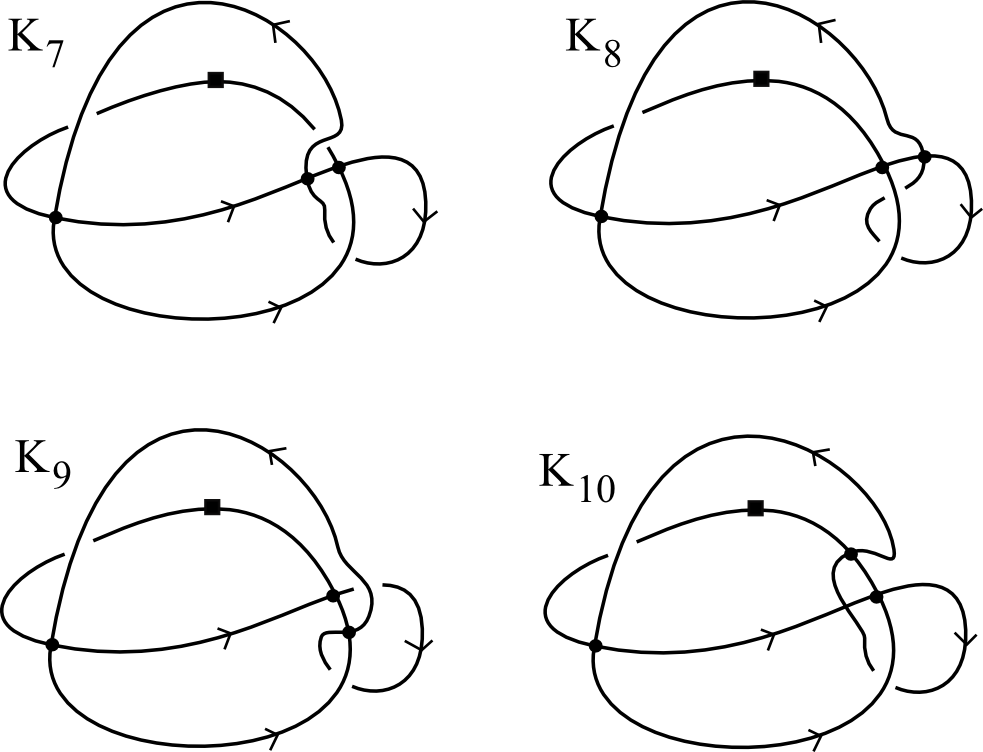}$$
\caption{Singular knots $K_7$, $K_8$, $K_9$, and $K_{10}$.}
\label{fig:k7to10}
\end{center}
\end{figure}

Since each of the four knots $K_3,\ldots,K_6$ has the same singularity data as one of $K_7,\ldots,K_{10}$, we have four pairs of knots which are isotopic in $\Imm_{=\beta_i}(\hat{S}^1,\R^4)$, and thus in $\Imm_{=\beta_i}(\hat{S}^1,\R^d)$ with $d\geq 4$, where $\beta_3,\ldots,\beta_6$ each encodes singularity data for a knot with exactly three double points.  If $d>4$, we require that the isotopy be through knots in $\R^4\subset \R^d$ (with the standard embedding).  If $d = 4$, we restrict the steps of the isotopy, as described in the Appendix, to simplify evaluation of Longoni cocycle on the cycle.   Resolving each singular knot in these four isotopies yields four families, denoted $\M_3, \M_4,\M_5,$ and  $\M_6$, parametrized by $S^{d-3}\times S^{d-3}\times S^{d-3}\times\I$.  Specifically, if $h_i:\I \rightarrow  \Imm_{\geq \beta}(\hat{S}^1,\R^d)$ is an isotopy, then these $\M_i$ are be the images of the composites
\begin{center} \begin{multline}\begin{CD} M_i\times \I = S^{d-3}\times S^{d-3}\times S^{d-3}\times\I  @>{\mathrm{Id}\times h_i}>> \\ S^{d-3}\times S^{d-3}\times S^{d-3}\times\Imm_{= \beta_i}({S}^1,\R^d) @>{\rho_{\beta_i,S_i}}>> \Emb(\hat{S}^1,\R^d). \end{CD}\end{multline}  \end{center}
For $i=3,4,5,6$, the boundary of $\M_i$ is the disjoint union of a boundary component of $\M_1$ and a boundary component of $\M_2$, providing a way to glue the boundary of $\M_1$ to the boundary of $\M_2$.  

The union of these six $(3d-8)$-dimensional families in $\Emb(\hat{S}^1,\R^d)$ gives a single family without boundary.  Let  \[M_\beta = \left(M_1\sqcup M_2 \sqcup\left( \sqcup_{i=3}^6 M_i\times \I\right) \right)/ \sim\] where each boundary component of $M_3,\ldots, M_6$ is identified with a boundary component of $M_1$ or $M_2$ so as to be compatible with Proposition ~\ref{prop:resolboundary}.  Let $\M_\beta$ be the image of the orientable topological manifold $M_\beta$ under the resolution map defined above.  For $d=4$, the resolution map takes $M_\beta$ to $\Emb(S^1,\R^d)$, as the isotopies we have chosen do not respect the fixed basepoint. 

\begin{theorem} If $d>4$ is even then the fundamental class of $\M_\beta$ is a non-trivial homology class in $\Emb(\hat{S}^1,\R^d)$ for any choice of isotopies $h_i$ through $\Imm_{=\beta_i}(\hat{S}^1,\R^d)$.  For $d=4$ the fundamental class of $\M_\beta$ is a non-trivial homology class in $\Emb(S^1,\R^d)$ if the isotopies $h_i$ satisfy a sequence of specified steps.  
\end{theorem}

For more details on the case $d=4$, see the Appendix and \cite{Pela12}.   To prove $[\M_\beta]$ is nontrivial, we evaluate a cocycle due to Longoni \cite{Long04} on $[\M_\beta]$ using configuration space integrals.  This is the main result of Section 4.

\section{The Longoni cocycle}

In \cite{CCL02},  Cattaneo, Cotta-Ramusino, and Longoni use configuration space integrals to define a chain map $I$ from a complex of decorated graphs to the de Rham complex of $\Emb({S}^1,\R^d)$.  The starting point is the evaluation map $ev: C_q[S^1]\times \Emb(S^1,\R^d)\rightarrow C_q[\R^d]$, where $C_q[M]$ is the Fulton-MacPherson compactified configuration space.  See \cite{Sinh04} for more details.  For some graphs $G$ (namely those with no internal vertices), the image of the chain map $I$ is defined by pulling back a form determined by $G$ from $C_q[\R^d]$ to $C_q[S^1]\times\Emb(S^1,\R^d)$ and then pushing forward to $\Emb(S^1,\R^d)$.  

To understand the general case, let $\Conf_{q,r}[\R^d]$ be the total space of the pull-back bundle shown below: \[ \xymatrix{\Conf_{q,r}[\R^d] \ar[r]^{\hat{ev}}  \ar[d] & C_{q+r}[\R^d] \ar[d] \\ C_q^{ord}[S^1]\times \Emb({S}^1,\R^d) \ar[r]^{\qquad ev} & C_q[\R^d] \; ,}\] where $C_q^{ord}[S^1]$ is the connected component of $C_q[S^1]$ in which the ordering on the points in the configuration agrees with the ordering induced by the orientation of $S^1$.  Fix an antipodally symmetric volume form on $S^{d-1}$, denoted $\alpha$.  A choice of $\alpha$ determines tautological $(d-1)-$forms on $\Conf_{q,r}[\R^d]$, defined by \[ \theta_{ij} = \hat{ev}^* \phi_{ij}^*(\alpha)\] where $\phi_{ij}: C_q(\R^d)\rightarrow S^{d-1}$ sends a configuration to the unit vector from the $i-$th point to the $j-$th point in the configuration.  We use integration over the fiber of the bundle $\Conf_{q,r}[\R^d] \rightarrow \Emb(S^1,\R^d)$, which is the composite of the projections \[ \Conf_{q,r}[\R^d]\rightarrow C_{q}^{ord}[S^1]\times \Emb({S}^1,\R^d) \rightarrow \Emb({S}^1,\R^d),\] to push forward products of the tautological forms to forms on $\Emb({S}^1,\R^d)$.  Which forms to push forward will be determined by graphs.

Consider connected graphs which satisfy the following conditions.   A \textit{decorated graph (of even type)} is a connected graph consisting of an oriented circle, vertices on the circle (called \textit{external vertices}), vertices which are not on the circle (called \textit{internal vertices}), and edges.  We require that all vertices are at least trivalent.  The decoration consists of an enumeration of the edges and an enumeration of the external vertices that is cyclic with respect to the orientation of the circle.  We will call the portion of the oriented circle between two external vertices an \textit{arc}.

\begin{definition} Let $\D_e$ be the vector space generated by decorated graphs of even type with the following relations.  We set $G = 0$ if there are two edges in $G$ with the same endpoints, or if there is an edge in $G$ whose endpoints are the same internal vertex.  The graphs $G$ and $G'$ are equal if they are isomorphic as graphs and the enumerations of their edges differ by an even permutation.
\end{definition}
The vector space $\D_e$ admits a bigrading as follows.  Let $v_e$ and $v_i$ be the number of external and internal vertices, respectively, and let $e$ be the number edges.  The order of a graph is given by \[\mathrm{ord}\, G = e - v_i\] and the degree of a graph is defined by \[\deg G = 2e - 3v_i -v_e.\]  Let $\mathcal{D}_e^{k,m}$ be the vector space of equivalence classes with order $k$ and degree $m$.  In \cite{CCL02}, Cattaneo, Cotta-Ramusino and Longoni define a map from this vector space to the space of ${(m+(d-3)k)-}$ forms on $\Emb(S^1, \R^d)$.  

\begin{definition} Define $I(\alpha): \mathcal{D}_e^{k,m} \rightarrow \Omega^{m+(d-3)k} \left(\Emb({S}^1,\R^d)\right)$ as follows.  \begin{enumerate}
\item Choose an ordering on the internal vertices.  
\item Associate each edge in $G$ joining vertex $i$ and vertex $j$ to the tautological form $\theta_{ij}$.  
\item Take the product of these tautological forms with the order of multiplication determined by the enumeration of the edges, to define a form on $\Conf_{q,r}[\R^d]$.
\item Integrate this form over the fiber to obtain a form on $\Emb(S^1,\R^d)$. \end{enumerate} 
\end{definition}

This integration over the fiber defines the pushforward and in this case is often called a configuration space integral. There is a coboundary map on $\D_e$ which makes $I(\alpha)$ a cochain map.

\begin{definition} Define a coboundary operator on  $\D_e$ by taking $\delta G$ to be the signed sum of the decorated graphs obtained from $G$ by contracting, one at a time, the arcs of $G$ and the edges of $G$ which have at least one endpoint at an external vertex.  After contracting, the edges and vertices are relabeled in the obvious way - if the edge (respectively vertex) labeled $i$ is removed, we replace the label $j$ by $j-1$ for all $j>i$.  When contracting an arc joining vertex $i$ to $i+1$, the sign is given by $\sigma(i,i+1) = (-1)^{i+1}$, and when contracting the arc joining vertex $j$ to vertex $1$, the sign is given by $\sigma(j,1) = (-1)^{j+1}$.  When contracting the edge $l$, the sign is given by $\sigma(l) = l + 1 + v_e$, where $ v_e$ is the number of external vertices.  
\end{definition}

\begin{theorem} \cite{CCL02}
The map $I(\alpha)$ determines a cochain map and therefore induces a map on cohomology, which we denote $I(\alpha): H^{k,m}(\mathcal{D}_e)\rightarrow H^{m+(d-3)k}(\Emb(\hat{S}^1,\R^d))$.\end{theorem}

At the level of forms, $I(\alpha)$ depends on the choice of antipodally symmetric volume form $\alpha$.   On cohomology, when $d > 4$ this is independent of $\alpha$.  

\begin{example}  From \cite{CCL02}, we have the graph cocycle shown in  Figure \ref{fig:ccl_cocycle}, originally investigated by Bott and Taubes \cite{BoTa94} for $d=3$.  
\begin{figure}[h]
\begin{center}
$$\includegraphics[width=2.25in]{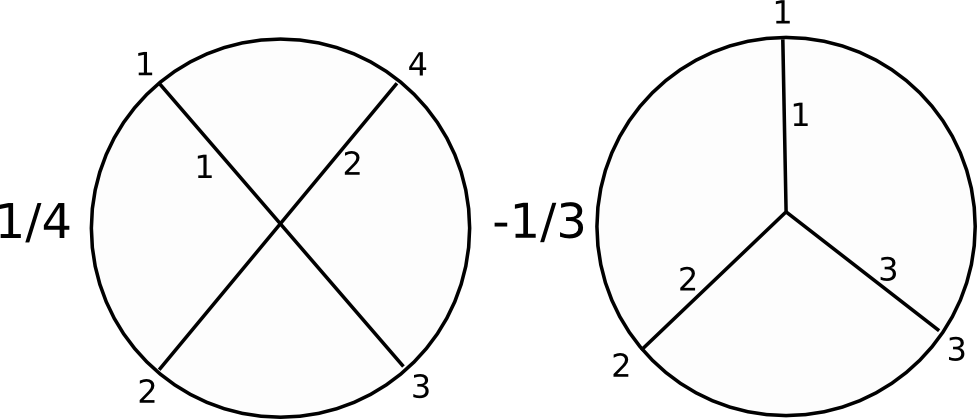}$$
\caption{Graph cocycle given by Cattaneo et al. in \cite{CCL02}.}
\label{fig:ccl_cocycle}
\end{center}
\end{figure}

This induces the cocycle \[ \frac14 \int_{\Conf_{4,0}[\R^d]}\theta_{13}\theta_{24} - \frac13\int_{\Conf_{3,1}[\R^d]}\theta_{14}\theta_{24}\theta_{34} \in H^{2d - 6}\left(\Emb(\hat{S}^1,\R^d)\right).\] 
\end{example}

In \cite{CCL02}, Cattaneo et al. show that this cocycle evaluates non-trivially on $\rho_{[x_1,x_3]\cdot[x_2,x_4]}\left(K\times S^{d-3}\times S^{d-3}\right),$ where $K$ is a singular knot with two double points respecting $[x_1,x_3]\cdot[x_2,x_4]$ (in this case, the cycle does not depend on the ordered subset $S\subseteq \{x_1,x_2,x_3,x_4\}$).  

\begin{example}  In \cite{Long04}, Longoni gives the example shown in Figure \ref{fig:longoni_cocycle} of a graph cocycle $G_L$ in $H^{3,1}(\D_e)$ which uses nontrivalent graphs. There $I(\alpha)\left(G_L\right)\in H^{3(d-3)+1}(\Emb(\hat{S}^1,\R^d))$ is the form \[\omega = \int_{\Conf_{4,1}[\R^d]} \theta_{15}\theta_{45}\theta_{35}\theta_{25} + 2\int_{\Conf_{5,0}[\R^d]} \theta_{13}\theta_{14}\theta_{25}.\]   We pair this cocycle with the cycle $[M_\beta]$ defined in Section ~2 to see that both are nontrivial. 
\begin{figure}[h]
\begin{center}
$$\includegraphics[width=2.5in]{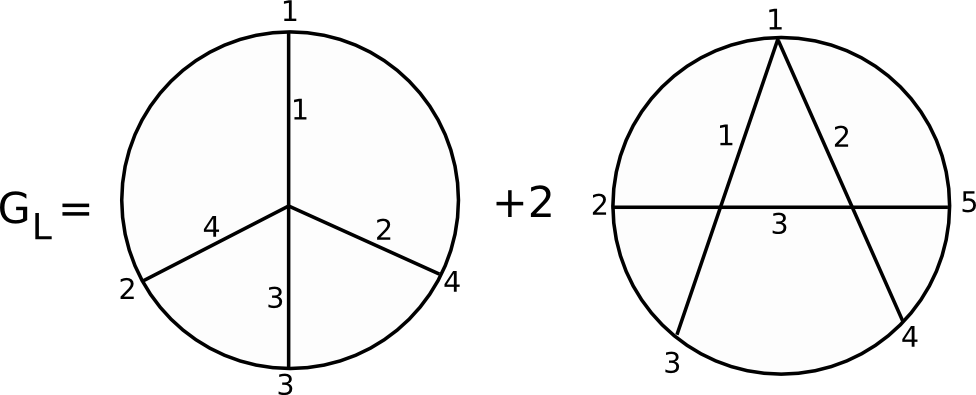}$$
\caption{Graph cocycle given by Longoni in \cite{Long04}.}
\label{fig:longoni_cocycle}
\end{center}
\end{figure}    
\end{example}

\section{Nontriviality}  

\begin{proposition}  Assume $d> 4$ is even.  Let $[\M_\beta]\in H_{3(d-3)+1}(\Emb(\hat{S}^1,\R^d))$ be the cycle defined in Section ~\ref{sec:cycledef}, and let $\omega\in H^{3(d-3)+1}(\Emb(\hat{S}^1,\R^d))$ be the Longoni cocycle defined in the last section.  Then $\omega([\M_\beta]) = \pm 2$.   In particular, $\omega([\M_\beta])$ is nonzero, and therefore both $\omega$ and $[\M_\beta]$ are non-trivial. 
\end{proposition}

  In \cite{Tour07}, Turchin calculates that $E^2_{-5,3(d-1)}$ has rank one, so $[\M_\beta]$ is a generator of this group.  The proposition also holds for $d=4$ if $\Emb(\hat{S}^1,\R^d)$ is replaced by $\Emb(S^1,\R^d)$.

\begin{proof}

First we show that $\omega_2([\M_\beta])=\pm 1$.  Let $g: \Conf_{5,0}[\R^d]\rightarrow S^{d-1}\times S^{d-1}\times S^{d-1}$ be the map shown in the diagram below, where $\bar{\psi}=\phi_{13}\times \phi_{14}\times \phi_{25}$.  Then $\omega_2$ is the pushforward along $\pi:\Conf_{5,0}[\R^d]\rightarrow \Emb(\hat{S}^1,\R^d)$ of $g^*(\alpha\otimes\alpha\otimes\alpha)$.  

\[\xymatrix{ &\Conf_{5,0}[\R^d]\ar[r] \ar[d] \ar@/^2pc/[rr]^g\ar@/_6pc/[dd]_\pi & C_5[\R^d] \ar[r]^{\bar{\psi}\qquad \qquad}\ar[d]^{id} & S^{d-1}\times S^{d-1}\times S^{d-1} \\ & C_5^{ord}[S^1]\times \Emb(\hat{S}^1,\R^d)\ar[r]\ar[d]  & C_5[\R^d] \\ \M_\beta \ar@^{(->}[r] & \Emb(\hat{S}^1,\R^d)   }\]

By naturality of pushforwards, $\omega_2([\M_\beta])=g^*(\alpha\otimes\alpha\otimes\alpha)([\pi^{-1}(\M_\beta)])$.  The bundle  $\pi: \Conf_{5,0}[\R^d]\rightarrow \Emb(\hat{S}^1,\R^d)$ is trivial, so $g^*(\alpha\otimes\alpha\otimes\alpha)([\pi^{-1}(\M_\beta)])= \int_{C_5^{ord}[S^1]\times \M_\beta} g^*(\alpha\otimes\alpha\otimes\alpha)$.

To calculate $\int_{C_5^{ord}[S^1]\times \M_\beta} g^*(\alpha\otimes\alpha\otimes\alpha)$, we first partition $C_5^{ord}[S^1]$.  For $i=1,\ldots, 5$ let $N_i=(t_i-\varepsilon, t_i+\varepsilon)$, where the $t_i$ are the times of singularity in $K_1$ and $K_2$, and $\varepsilon$ is as in Section ~\ref{sec:cycledef}.    Define \[ C_5^{(i)} = \left\{ \bar{s}\in C_5^{ord} : s_j\not\in N_i\; \mathrm{for}\; j=1,\ldots, 5 \; \mathrm{and}\; \bar{s}\notin C_5^{(m)} \; \mathrm{for} \;m<i\right\},\]  and $C_5^c = C_5^{ord}[S^1]\backslash\left( \cup_{i=1}^5 C_5^{(i)}\right)$, so $C_5^c$ is the set of all $\bar{s}\in C_5^{ord}[S^1]$ such that $t_i-\varepsilon < s_i< t_i+\varepsilon$ for $i=1,\ldots, 5$.  Then $C_5^{ord}[S^1]$ decomposes as $C_5^{ord}[S^1]= C_5^c\sqcup C_5^{(1)}\sqcup \cdots \sqcup C_5^{(5)}$, and we obtain a corresponding decomposition of $\int_{C_5^{ord}[S^1]\times \M_\beta} g^*(\alpha\otimes\alpha\otimes \alpha)$.  We will show that $\int_{C_5^{(m)}\times \M_\beta}g^*(\alpha\otimes\alpha\otimes\alpha) = 0$ for $m=1,\ldots, 5$, so calculating $\omega_2([\M_\beta])$ reduces to evaluating the integrals \[\int_{C_5^c\times \M_i} g^*(\alpha\otimes\alpha\otimes\alpha).\]

For $m=3,4,5$, we show $\int_{C_5^{(m)}\times \M_\beta}g^*(\alpha\otimes\alpha\otimes\alpha) = 0$ by showing $\int_{C_5^{(m)}\times \M_i}g^*(\alpha\otimes\alpha\otimes\alpha) = 0$ for $i = 1,\ldots,6$.   Recall that manifolds have only trivial forms in degrees above their dimension, so a form pulled back through a smaller dimensional manifold is always zero.  To prove that the integrals $\int_{C_5^{(m)}\times \M_i}g^*(\alpha\otimes\alpha\otimes\alpha) $ are zero, we show that the map $g$ factors through spaces of smaller dimension when restricted to each of the subspaces $C_5^{(m)}\times \M_i$.  

First, consider the case $\int_{C_5^{(3)}\times \M_1} g^*(\alpha\otimes\alpha\otimes\alpha)$.  Recall that  $\M_1$  is $\rho_{\beta_1,S}\left(K_1\times \prod_{k=3}^5 (v_k,a_k,\varepsilon)\right) $.  If $t\notin N_3$ and $\gamma \in \M_1$, the point $\gamma(t)$ does not depend on the value of $v_3$ in the preimage of $\gamma$.    This gives us the following factorization of $g\big|_{C_5^{(3)}\times \M_1}$: 

\[ \xymatrix{C_5^{(3)}\times \M_1 \ar[rr]^g \ar[dr] & & S^{d-1}\times S^{d-1}\times S^{d-1} \\ & C_5^{(3)}\times S^{d-3}\times S^{d-3} \ar[ur] & } \]

Since $\dim(C_5^{(3)}\times S^{d-3}\times S^{d-3} ) = 2d-1$ is less than $\dim(S^{d-1}\times S^{d-1}\times S^{d-1}) = 3d-3$, we have $\int_{C_5^{(3)}\times \M_1}g^*(\alpha\otimes\alpha\otimes\alpha)=0$.   

Similarly, for $m=4$ or $m=5$, the restriction $g\big|_{C_5^{(m)}\times\M_1}$ factors through \[C_5^{(m)} \times \{v_3\in S^{d-2}\,:\, \parallel v_3\pm w_1 \parallel > \delta\; \mathrm{and} \parallel v_3\pm w_4\parallel >\delta\}\times S^{d-3},\] so the corresponding integrals are zero.   This argument also shows that $\int_{C_5^{(m)}\times \M_2} g^*(\alpha\otimes\alpha\otimes\alpha) = 0$ for $m=3,4,5$.   For $i=3,4,5,6$ and $m=3,4, 5$, the restriction $g\big|_{C_5^{(m)}\times \M_i}$ factors through $S^{d-3}\times S^{d-3}\times \I$ and therefore $\int_{C_5^{(m)}\times \M_i}g^*(\alpha\otimes\alpha\otimes\alpha)$ is zero.  
We show $\int_{C_5^{(1)}\times \M_\beta} g^*(\alpha\otimes\alpha\otimes\alpha) = 0$ by replacing $\M_\beta$ with the family of embeddings obtained by moving the first strand (instead of the fourth) off of the double point $K_i(t_1) =K_i(t_4)$, over which $g^*$ factors through a space of lower dimension.  We replace $\M_\beta$ in two steps - first with the family of embeddings in which both strands are moved off the double point, and then by the family in which only the first strand is moved.

Let $\M_\beta'$ be the piecewise smooth subspace of $\Emb(\hat{S^1}, \R^d)$ defined similarly to $\M_\beta$, but by choosing the ordered subset of variables in $\beta_1$ and $\beta_2$ to be $S=\{x_1,x_3,x_4,x_5\}$, and fixing $a_1=a_4$ and $v_1=-v_4$.  In other words, $\M_\beta'$ is obtained from $K_1,\ldots, K_6$ by moving both strands off the double point $K_i(t_1)=K_i(t_4)$ in antipodal directions.  

We define a cobordism $W_1$ between $\M_\beta$ and $\M_\beta'$ as the subspace of $\Emb(\hat{S^1}, \R^d)$ parametrized by $\left(\sqcup_i M_i\right)\times \I$, with the embedding corresponding to the parameter $u\in \I$ determined by $a_1 = u a_4$ (so the $\I$ parametrizes how far the strand with $K_i(t_1)$ is moved off the double point).  

By Stokes' theorem, \[  \int_{C_5^{(1)}\times W_1} dg^*(\alpha\otimes\alpha\otimes\alpha) \\ = \int_{\partial(C_5^{(1)}\times W_1)} g^*(\alpha\otimes\alpha\otimes\alpha).\] Since $dg^*(\alpha\otimes\alpha\otimes\alpha) =g^*d(\alpha\otimes\alpha\otimes\alpha) = 0$, we have  \begin{equation}\label{stokes} 0 = \int_{\partial C_5^{(1)}\times W_1} g^*(\alpha\otimes\alpha\otimes\alpha) +\int_{C_5^{(1)}\times \M_\beta} g^*(\alpha\otimes\alpha\otimes\alpha)-\int_{C_5^{(1)}\times \M_\beta'} g^*(\alpha\otimes\alpha\otimes\alpha).\end{equation}  The restriction $g^*\big|_{\partial C_5^{(1)}\times W_1}$ factors through $\partial C_5^{(1)}\times \left(\sqcup_i \M_i\right)$.  If $\bar{s}\in \partial C_5^{(1)}$ then  the parameter, $u\in \I$ determining how far the first strand is moved does not affect $g(\bar{s}, \gamma)$ for $\gamma\in W_1$.  Thus, $\int_{\partial C_5^{(1)}\times W_1} g^*(\alpha\otimes\alpha\otimes\alpha) = 0$ and \[ \int_{C_5^{(1)}\times \M_\beta} g^*(\alpha\otimes\alpha\otimes\alpha)=\int_{C_5^{(1)}\times \M_\beta'} g^*(\alpha\otimes\alpha\otimes\alpha).\]

Let $\M_\beta''$ be the piecewise smooth subspace of $\Emb(\hat{S^1}, \R^d)$ obtained by choosing the ordered subset of variables in $\beta_1$ and $\beta_2$ to be $S=\{x_1,x_3,x_5\}$.  In other words, $\M_\beta''$ is obtained from $K_1,\ldots, K_6$ by moving only the first strand off the double point $K_i(t_1)=K_i(t_4)$.  Let $W_2\subset \Emb(\hat{S^1}, \R^d)$ be parametrized by $\left(\sqcup_i M_i\right)\times \I$, with the embedding corresponding to the parameter $u\in \I$ given by choosing $a_4'' = u a_4$ (so the interval parametrizes how far the strand with $K_i(t_4)$ is moved off the double point).  Then $W_2$ gives a cobordism between $\M_\beta'$ and $\M_\beta''$, as $\partial W_2 = \M_\beta \sqcup (-\M_\beta')$.  Using Stokes' Theorem and naturality again, we have \begin{equation}\label{stokes2} 0 = \int_{\partial C_5^{(1)}\times W_2} g^*(\alpha\otimes\alpha\otimes\alpha) +\int_{C_5^{(1)}\times \M_\beta'} g^*(\alpha\otimes\alpha\otimes\alpha)-\int_{C_5^{(1)}\times \M_\beta''} g^*(\alpha\otimes\alpha\otimes\alpha).\end{equation}  

The restriction $g^*\big|_{\partial C_5^{(1)}\times W_2}$ does not factor through $\partial C_5^{(1)}\times \left(\sqcup_i M_i\right)$.  To show the first integral in \eqref{stokes2} is zero, we consider $W_2$ as a subspace of $\Imm_{\leq [x_1,x_2], t_1,t_4}(\hat{S}^1,\R^d)$, the subset of $\Imm(\hat{S}^1,\R^d)$ consisting of all immersions $\gamma$ with at most one singularity - a double point $\gamma(t_1)=\gamma(t_4)$.  Since a configuration in $\partial C_5^{(1)}$ does not contain the point $t_1$, the map $g$  is well-defined on $\partial C_5^{(1)}\times \mathrm{Imm}_{\leq [x_1,x_2],t_1,t_4}(\hat{S}^1,\R^d)$.  Letting the dependance on the lengths of the strands be apparent, we now work with $W_2=W_2(\ep_3,\ep_4,\ep_4)$ as a subspace of $\mathrm{Imm}_{\leq [x_1,x_2],t_1,t_2}(\hat{S}^1,\R^d)$.  In this larger space, $W_2(\ep_3,\ep_4,\ep_5)$ is cobordant to $W_2(\ep_3,0,\ep_5)$.  The cobordism is given by $W_3\subset \Imm_{[x_1,x_2],t_1,t_4}(\hat{S}^1,\R^d)$ parametrized by $\left(\sqcup_i M_i\right)\times \I\times \I$ where the second unit interval parametrizes the length of the strand centered at $t_4$ moved by the resolution map.     

By Stokes' Theorem and naturality, \[ 0=\int_{\partial C_5^{(1)} \times W_3} dg^*(\alpha\otimes\alpha\otimes\alpha) = \int_{\partial\left(\partial C_5^{(1)} \times W_3\right)}g^*(\alpha\otimes\alpha\otimes\alpha), \]  and thus, \begin{multline} 0 = \int_{\partial (\partial C_5^{(1)})\times W_3}g^*(\alpha\otimes\alpha\otimes\alpha) + \int_{\partial C_5^{(1)}\times W_2(\ep_3,\ep_4,\ep_5)} g^*(\alpha\otimes\alpha\otimes\alpha) - \int_{\partial C_5^{(1)}\times W_2(\ep_3,0,\ep_5)} g^*(\alpha\otimes\alpha\otimes\alpha)\\ = \int_{\partial C_5^{(1)}\times W_2(\ep_3,\ep_4,\ep_5)} g^*(\alpha\otimes\alpha\otimes\alpha). \end{multline} The second equality holds because $\partial(\partial C_5^{(1)} )= \varnothing$ and the dimension of $W_2(\ep_3,0,\ep_5)$ is $2d-3$. 

By the same argument,  $\int_{C_5^{(5)}\times \M_\beta} g^*(\alpha\otimes\alpha\otimes\alpha)=0$.  Calculating $\int_{C_5\times \M_\beta }g^*(\alpha\otimes\alpha\otimes\alpha)$ thus reduces to calculating $\int_{C_5^c\times \M_i} g^*(\alpha\otimes\alpha\otimes\alpha)$ for $i=1,\ldots, 6$.   

We chose the antipodally symmetric volume form, $\alpha$, to be concentrated near the points $\bar x_1 = (0,\ldots,0,1)\in S^{d-1}$ and $\bar x_2 = (0,\ldots,0,-1)\in S^{d-1}$.  Let $\tau_{\bar x_1}$ and $\tau_{\bar x_2}$ be the Thom classes of these points, as defined in Section 6 of \cite{BoTu82}, so $\alpha = \frac{1}{2}\left(\tau_{\bar x_1} + \tau_{\bar x_2}\right)$.   Let $y$ be the arc in $S^{d-1}$ connecting $(0,\ldots,0,1)$ and $(0,\ldots,0,-1)$, defined as \[y = \left\{ \left(0,\ldots,0,\sqrt{1-s^2},s\right)\in S^{d-1} \; : \; -1\leq s\leq 1\right\}.\]  The Thom class $\tau_y$ of $y$ can be chosen so that $d\tau_y = \tau_{\bar x_1} - \tau_{\bar x_2} = 2(\tau_{\bar x_1} - \alpha)$.  

We have \begin{align*} \int_{C_5^c\times \M_i} g^*(\alpha\otimes\alpha\otimes\alpha-\tau_{\bar x_1}\otimes \alpha \otimes \alpha) &=  \int_{C_5^c\times \M_i} g^*\left( -\tfrac{1}{2} d\tau_y\otimes \alpha \otimes \alpha\right) \\ & = -\tfrac{1}{2} \int_{C_5^c\times \M_i} dg^*(\tau_y\otimes \alpha\otimes \alpha) \\ & =-\tfrac{1}{2}\int_{\partial (C_5^c\times \M_i)} g^*(\tau_y\otimes \alpha\otimes \alpha). \end{align*}

 If $(v_1,v_2,v_3) \in g\left(\partial (C_5^c\times \M_i)\right)$ at least one of the first two coordinates of $v_1$ is non-zero, but every $\bar x\in y\subset S^{d-1}$ has $x_1,x_2 = 0$. Thus, the sets $y$ and  $g\left(\partial (C_5^c\times \M_i)\right)$ are disjoint and $\int_{C_5^c\times \M_i} g^*(\tau_y\otimes\alpha\otimes\alpha)=0$, which means $\int_{C_5^c\times \M_i} g^*(\alpha\otimes\alpha\otimes\alpha) = \int_{C_5^c\times \M_i} g^*(\tau_{\bar x_1}\otimes \alpha \otimes \alpha)$.  By a similar argument, \[\int_{C_5^c\times \M_i} g^*(\alpha\otimes\alpha\otimes\alpha) = \int_{C_5^c\times \M_i} g^*(\tau_{\bar x_1}\otimes \tau_{\bar x_1} \otimes \tau_{\bar x_1}).\]  This integral can be calculated by counting the transverse intersections of $g(C_5^c\times \M_i)$ and $(\bar x_1, \bar x_1, \bar x_1)$ in $S^{d-1}\times S^{d-1}\times S^{d-1}$.  

Recall that \[g(\bar{s},\gamma)=\left( \frac{\gamma(s_3)-\gamma(s_1)}{\|\gamma(s_3)- \gamma(s_1)\|}, \frac{\gamma(s_4)-\gamma(s_1)}{\|\gamma(s_4)- \gamma(s_1)\|}, \frac{\gamma(s_5)-\gamma(s_2)}{\|\gamma(s_5)- \gamma(s_2)\|}\right).\]  Thus, we are counting the number of pairs $(\bar{s},\gamma)\in C_5^c\times \M_i$ for which \[ \frac{\gamma(s_3)-\gamma(s_1)}{\|\gamma(s_3)- \gamma(s_1)\|} = \frac{\gamma(s_4)-\gamma(s_1)}{\|\gamma(s_4)- \gamma(s_1)\|} = \frac{\gamma(s_5)-\gamma(s_2)}{\|\gamma(s_5)- \gamma(s_2)\|} = (0,\ldots,0,1).\]    For $\gamma \in\ M_\beta$, this is only possible if $s_i=t_i$ for $i=1,\ldots,5$.  

If $\gamma \in \M_1$, then  \[ \frac{\gamma(t_3)-\gamma(t_1)}{\|\gamma(t_3)- \gamma(t_1)\|} = \frac{\gamma(t_4)-\gamma(t_1)}{\|\gamma(t_4)- \gamma(t_1)\|} = \frac{\gamma(t_5)-\gamma(t_2)}{\|\gamma(t_5)- \gamma(t_2)\|} = (0,\ldots,0,1)\] exactly when $v_3=v_4=v_5 = (0,\ldots,0,1)$ and so $\int_{C_5^c\times \M_1} g^*(\alpha\otimes\alpha\otimes\alpha) = \pm1$.  If $\gamma \in \M_i$ for $i=2,\ldots, 6$, then \[\frac{\gamma(t_3)-\gamma(t_1)}{\|\gamma(t_3)- \gamma(t_1)\|} \neq (0,\ldots, 0, 1) ,\] and $\int_{C_5^c\times \M_i} g^*(\alpha\otimes\alpha\otimes\alpha) = 0$.  Thus, $\omega_2([\M_\beta]) = \pm 1$.  

Next, we show that $\omega_1([\M_\beta])=0$.  Let $f: \Conf_{4,1}(\R^d)\rightarrow  S^{d-1}\times S^{d-1}\times S^{d-1}\times S^{d-1}$ be the map shown in the diagram below, where $\bar \varphi = \phi_{15}\times \phi_{45}\times \phi_{35}\times \phi_{25}$.  Then $\omega_1$ is the pushforward of $f^*(\alpha\otimes \alpha\otimes \alpha\otimes\alpha)$ along $p: \Conf_{4,1}(\R^d)\rightarrow \Emb(\hat{S}^1,\R^d)$.  

\[\xymatrix{ &\Conf_{4,1}(\R^d)\ar[r] \ar[d]^{p_1} \ar@/^2pc/[rr]^f\ar@/_6pc/[dd]_p&C_5[\R^d] \ar[r]^{\bar{\varphi}\qquad \qquad}\ar[d] & \ssss \\ & C_4^{ord}[S^1]\times \Emb(\hat{S}^1,\R^d)\ar[r]\ar[d]^{p_2}  & C_4[\R^d] \\ M_\beta \ar@^{(->}[r] & \Emb(\hat{S}^1,\R^d)   }\]

Since $p^{-1}(\M_\beta) = p_1^{-1}(C_4^{ord}[S^1]\times \M_\beta)$, we have  $\omega_1([\M_\beta])=\int_{p_1^{-1}(C_4^{ord}[S^1]\times \M_\beta)}f^*(\aaaa)$.   Following the calculation of $\omega_1([\M_\beta])$, define \[ C_4^{(i)} = \left\{ \bar{s}\in C_4^{ord}[S^1] : s_j\not\in N_i\; \mathrm{for}\; j=1,\ldots, 4 \; \mathrm{and}\; \bar{s}\notin C_4^{(m)} \; \mathrm{for} \;m<i\right\}.\] Each configuration in $C_4^{ord}[S^1]$ has four points, so $C_4^{ord}[S^1] = C_4^{(1)}\sqcup\cdots\sqcup C_4^{(5)}$. The arguments used to prove that $\int_{C_5^{(m)} \times \M_\beta} g^*(\aaa) = 0$ also show $\int_{p_1^{-1}(C_4^{(m)}\times \M_\beta)}f^*(\aaaa) = 0$ for $m = 1,\ldots,5$.  

\end{proof}

\section{Future Work}

The resolution map in Definition ~\ref{def:resmap} can be generalized to define a resolution map for knots respecting any bracket expression.  Instead of choosing an ordered subset of the variables, we repeatedly choose the strands to move so as to resolve the singularity data for the brackets which are not contained inside of any other brackets.  

For example, if $(K, \bar{t})$ respects $[[x_1,x_3],[[x_2,x_4],x_5]]$ the point $K(t_1)=K(t_3)$ is first moved away from the point $K(t_2)=K(t_4)=K(t_5)$, turning the original singularity into a double point and a triple point.  The double point is then resolved as before, and the triple point is resolved by first moving the fifth strand off the singularity and then resolving the remaining double point.      

The description in Section ~\ref{sec:cycledef} of the first differential of the embedding calculus homology spectral sequence is given in terms of ``doubling'' the point $x_i$.  In \cite{Pela12} we develop another description of this differential, call it $\tilde{d}_1$, which encodes the singularity data that occurs when a knot respecting a bracket expression is resolved as prescribed in the generalization of the resolution map, but with the directions chosen in such a way as to introduce a new singularity.   The boundary components of the family of resolutions of a knot $(K,\bar{t})$ respecting a bracket expression under the generalized resolution map are the same as the families of resolutions of knots respecting the terms in $\tilde{d}_1$ of that bracket expression (with appropriate choices).

Suppose $\beta = \sum_{i=1}^m \beta_i$ is a cycle on the first page of the spectral sequence (where each $\beta_i$ is a bracket expression with a single term) in which the Jacobi identity is not used to simplify the differential.   Knots $(K_i, \bar{t})$ respecting the $\beta_i$ can be chosen so that the boundaries of the families of resolutions under the generalized resolution map can be connected by families of embeddings given by an isotopy of underlying singular knots, as in the cycle $[\M_\beta]$ defined here.  Thus the process used in this paper can be generalized to more cycles on the first page of the spectral sequence.  Because Turchin proved linear duality of the Cattaneo, Cotta-Ramusino, and Longoni graph complex and the $E_1$ page of the embedding calculus spectral sequence, we also have configuration space integrals to evaluate on the families we produce.  Together these could give not only a second proof of the collapse of the spectral sequence (Lambrechts, Turchin and Volic use closely related configuration space integrals in their proof of the collapse in \cite{LTV06}), but also geometric representatives and a clear starting point for considering any torsion phenomena.

\appendix 

\section{Isotopies}

When $d = 4$ the value of $\omega_2([\M_\beta])$ depends on the isotopies chosen, as $x_1 = (0,0,0,1)$.  We can construct isotopies whose images are in $\R^3$ except for near crossing changes.  This forces the counts used to calculate the integrals $\int_{C_5^c\times \M_i} g^*(\tau_{\bar x_1}\otimes \tau_{\bar x_1} \otimes \tau_{\bar x_1})$ to be the same as in the higher dimensional cases.  We give an example of such an isotopy from $K_3$ to $K_9$ below, by specifying steps the isotopy must satisfy.  All four isotopies will appear in \cite{Pela12}.

By a slide isotopy we will mean an isotopy through singular knots in which a singular point is moved along one of the strands through the singularity while the other strand moves along with the singular point.  By a planar isotopy we will mean an isotopy which can be represented by an isotopy of knot diagrams.  Isotopies corresponding to the Reidemeister moves in classical knot theory generalize to singular knots in $\R^d$.  In addition to the usual Reidemeister I and II moves, we use Reidemeister III moves to move a strand past a crossing (as in classical theory) or past a singularity, as shown in Figure \ref{fig:rIII}.   

\begin{figure}[h]
\begin{center}
$$\includegraphics[width=2in]{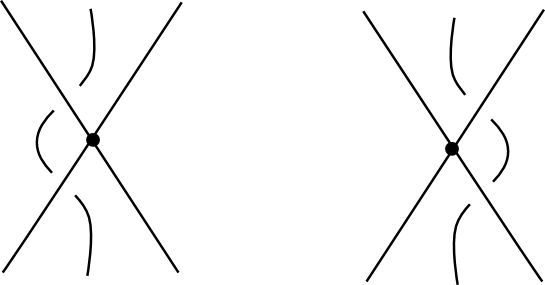}$$
\caption{Reidemeister III move for singular knots.}
\label{fig:rIII}
\end{center}
\end{figure}

By a ``rotate the disk isotopy," we mean an isotopy in which the disk centered at a singularity is rotated by $180^\circ$ about the axis perpendicular to a particular great circle.  Specifically, we take two distinct nested disks centered at the singular point with radii small enough that the intersection of the knot with the disks is the two strands intersecting at the singular point.  The smaller of the two disks is rotated by $180^\circ$ without changing anything inside of this disk.  The strands inside of the larger disk but outside of the smaller disk are stretched through a planar isotopy.  This isotopy is shown in Figure ~\ref{fig:rotate_isotopy} from the perspective of the north pole of the larger disk.  The knot remains unchanged outside of the larger disk.

\begin{figure}[h]
\begin{center}
$$\includegraphics[width=2.5in]{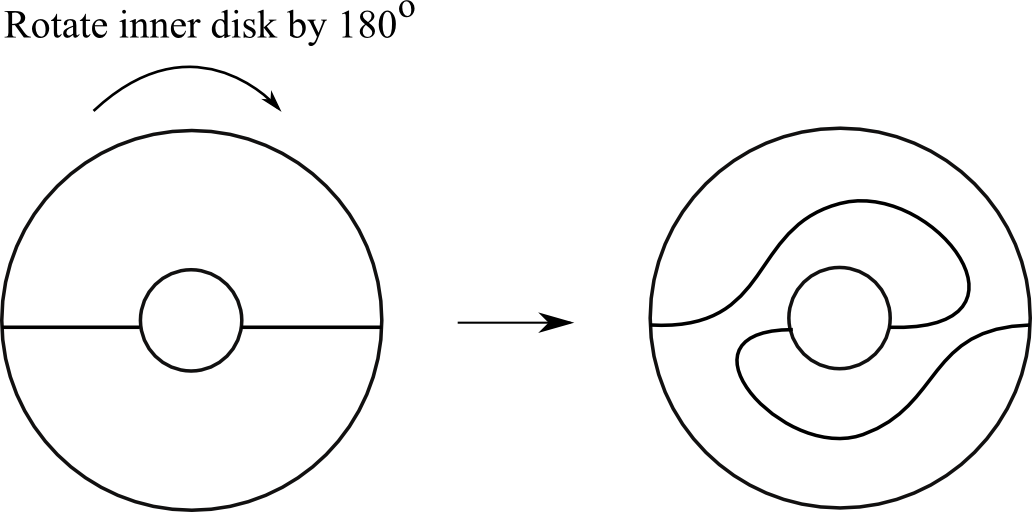}$$
\caption{View of a rotate the disk isotopy from the north pole.}
\label{fig:rotate_isotopy}
\end{center}
\end{figure}

A suitable type of isotopy from $K_3$ to $K_{9}$ is shown in Figure \ref{fig:4isotopic10}, and the steps are given below.  Each step occurs in $\R^3\subset \R^4$ except (4), (6) and (10), in which one strand of the knot briefly moves into $\R^4$. \begin{enumerate}
\item Simplify the shape of the strand from $b_1$ to $c_1$ and perform a Reidemeister II move on the strand from $a_1$ to $b_1$ to eliminate crossings.
\item Move the points $a_1$, $b_1$ and $c_1$ to $a_2$, $b_2$ and $c_2$ through a planar isotopy.
\item Rotate the disk centered at  $c_2$ by $180^\circ$ about the axis perpendicular to the great circle shown.
\item The crossing is changed, briefly moving the strand from $b_2$ to $c_2$ in the direction of the fourth standard basis vector. 
\item Perform a sequence of Redemeister I,II and III moves on the strand from $b_2$ to $c_2$.
\item The crossing is changed, briefly moving the strand from $c_2$ to $a_2$ in the direction of the fourth standard basis vector.  
\item Perform a sequence of Reidemeister I, II and III moves on the strand from $c_2$ to $a_2$.  
\item Rotate the disk centered at  $a_2$ by $180^\circ$ about the axis perpendicular to the great circle shown.
\item Perform a sequence of Reidemeister I, II and III moves on the strand from $c_2$ to $a_2$ and the strand from $a_2$ to $b_2$.  
\item The crossing is changed, briefly moving the strand from $a_2$ to $b_2$ in the direction of the fourth standard basis vector. 
\item Perform a sequence of Reidemeister I, II and III moves on the strand from $a_2$ to $b_2$. 
\item  Through a planar isotopy, the points $a_2$, $b_2$ and $c_2$ are moved to the positions of the double points of $K_{9}$, denoted $a_3$, $b_3$ and $c_3$ and the strands are moved to give the knot the same shape as $K_{9}$.
\end{enumerate}

\begin{figure}[h]
\begin{center}
$$\includegraphics[width=4.75in]{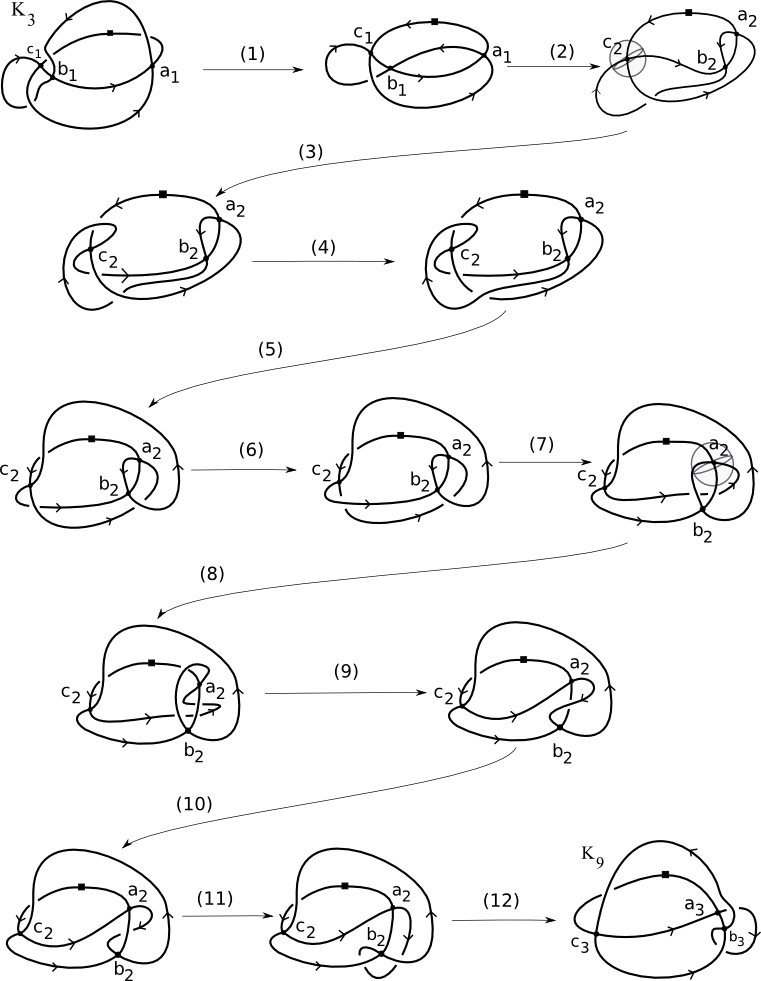}$$
\caption{Isotopy from $K_3$ to $K_{9}$.}
\label{fig:4isotopic10}
\end{center}
\end{figure}

\bibliographystyle{plain}
\bibliography{myrefs}  
\end{document}